\documentclass[11pt]{amsart}
\usepackage{amssymb, amsthm, graphicx, epsfig, verbatim, amscd,
enumerate, doublespace, stmaryrd, amsmath, psfrag}

\newtheorem{thm}{Theorem}[section]

\newtheorem{lem}[thm]{Lemma}

\theoremstyle{definition}
\newtheorem{defn}[thm]{Definition}

\theoremstyle{remark}
\newtheorem{rem}[thm]{Remark}

\numberwithin{equation}{section}

\newcommand{\al}{\alpha}
\newcommand{\be}{\beta}
\newcommand{\ga}{\gamma}

\newcommand{\de}{\delta}
\newcommand{\ep}{\varepsilon}

\newcommand{\la}{\lambda}

\newcommand{\ro}{\rho}

\newcommand{\Si}{\Sigma}

\newcommand{\va}{\varphi}
\newcommand{\csi}{\xi}

\newcommand{\x}{\times}

\newcommand{\CC}{\mathcal C}

\newcommand{\imm}{{\mathrm {Imm}}}

\newcommand{\Z}{\mathbb Z}

\newcommand{\R}{\mathbb R}

\newcommand{\RP}{{\mathbb R}{P}}

\newcommand{\del}{\partial}

\newcommand{\co}{\colon\thinspace}

\begin{document}
\mathsurround=1pt 
\title{Fold cobordisms and stable homotopy groups}


\thanks{The author is supported by Canon Foundation in Europe}

\author{Boldizs\'ar Kalm\'{a}r}

\address{Kyushu University, Faculty of Mathematics, 6-10-1 Hakozaki, Higashi-ku, Fukuoka 812-8581, Japan}

\email{kalmbold@yahoo.com}

\begin{abstract}
We prove that for $n \geq 1$ and $q > 0$ the (oriented) cobordism group
of fold maps of (oriented) $(n+q)$-dimensional manifolds into $\R^n$
contains the direct sum  
of $\lfloor (q+1)/2 \rfloor$ copies of the
$(n-1)$th stable homotopy group of spheres
as a direct summand.
We also prove that for $k \geq 1$ and $q = 2k -1$ the cobordism group of fold maps
of unoriented $(n+q)$-dimensional manifolds into $\R^n$ also 
contains the $n$th
stable homotopy group of the space $\RP^{\infty}$ as
a direct summand. 
We have the analogous results about bordism groups of fold maps as well.
\end{abstract}


\maketitle

\begin{spacing}{1.3}

\section{Introduction}\label{s:intro}

In the study of singular maps one basic approach is classifying singular maps 
under the equivalence relation {\it cobordism} and describing {\it cobordism groups} of singular maps
\cite{An, Ik, Ko, RSz, Szucs1, Sz1, Sz2, Szucs2, Sz3, Szucs4}.

This paper is about fold maps of 
$(n+q)$-dimensional manifolds into oriented $n$-dimensional manifolds. 
Fold maps of $(n+q)$-dimensional manifolds into $n$-dimensional manifolds have the formula 
$
f(x_1, \ldots, x_{n+q})=(x_1, \ldots, x_{n-1}, \pm x_n^2 \pm \cdots \pm x_{n+q}^2 )
$
as a local form around each singular point, and  
the subset of the singular points in the source manifold is a $(q+1)$-codimensional submanifold.
Moreover, if we restrict a fold map to the set of its singular points, then we obtain a
codimension one immersion into the target manifold of the fold map.
These properties give the possibility to define simple geometrical invariants of 
cobordisms of 
fold maps (see Definition~\ref{cobdef}) via
immersions of the singular sets.
And by \cite{We} these geometrical invariants show a strong relation with 
stable homotopy groups of Thom spaces of vector bundles like the 
circle $S^1$ and the projective space
$\RP^{\infty}$.
Furthermore, it turns out that these geometrical invariants detect 
an important torsion part of the cobordism group of fold maps,
namely the stable homotopy groups of spheres and the space $\RP^{\infty}$.

In \cite{An} Ando showed that the cobordism group of fold maps
of oriented $n$-dimensional manifolds into the $n$-dimensional 
Euclidean space $\R^n$ is isomorphic
to the $n$th stable homotopy group of spheres. 
In this paper we are looking for similar results in the case of
the cobordism group
of fold maps of $(n+q)$-dimensional manifolds into $\R^n$, where
$q > 0$.
The case $n=1$ (i.e., the cobordism group of Morse functions) has been computed by Ikegami \cite{Ik}. 

By recent results of Ando \cite{An2}, Sadykov \cite{Sad2} and Sz\H{u}cs \cite{Szucs4},
the cobordism groups of singular maps are equal to homotopy groups of spectra.
Moreover, the rank, the torsion part for sufficiently high primes (and
estimations for arbitrary primes) of the homotopy groups of these spectra can be computed 
in special and important cases in positive codimension \cite{Szucs4}
and very probably in negative codimension \cite{Kal3, Sad3}.
Our results in this paper show that computing the $p$-torsion of 
the cobordism groups of singular maps could be extremely difficult and hopeless for
a small prime $p$.
We note that in special dimension pairs \cite{EkSzuTer} computes the cobordism groups of
fold maps completely.

The paper is organized as follows.
In Section~\ref{s:jelol} we give several basic definitions and notations.
In Section~\ref{mainthms} we state our main results.
In Section~\ref{mainthmproof} we prove our main theorems.
In Section~\ref{bord} we give analogous results about bordism groups of fold maps.

The author would like to thank Prof. Andr\'as Sz\H{u}cs and Prof. Osamu Saeki for
the helpful discussions and corrections, and
 for suggesting to the author some special Morse functions.

\subsection{Notations}
In this paper the symbol ``$\amalg$'' denotes the disjoint union, 
for any number $x$ the symbol ``$\lfloor x \rfloor$'' denotes the greatest
integer $i$ such that $i \leq x$,
$\ga^1$ denotes the universal line bundle over $\RP^{\infty}$ 
and the symbol $\ep^1$ denotes the trivial line bundle over a point.
The symbol $\imm^{\csi^k}(n-k,k)$ denotes
the cobordism group of $k$-codimensional immersions into $\R^{n}$
whose normal bundles can be induced from the vector bundle $\csi^k$ (this
group is isomorphic to $\pi_{n}^s(T\csi^k)$ where $T\csi^k$ denotes 
 the Thom space of the bundle $\csi^k$ \cite{We}).
The symbol $\imm(n-k,k)$ denotes
the cobordism group $\imm^{\ga^k}(n-k,k)$ where
$\ga^k$ is the universal bundle for $k$-dimensional real vector bundles.
The symbol $\pi_n^s(X)$ ($\pi_n^s$) denotes the $n$th stable homotopy group of the space $X$ (resp. spheres).
The symbol ``id$_A$'' denotes the identity map of the space $A$.
The symbol $\ep$ refers to a small positive number.
All manifolds and maps are smooth of class $C^{\infty}$.

\section{Preliminaries}\label{s:jelol}

\subsection{Fold maps}
 
Let $n \geq 1$ and $q > 0$.
Let $Q^{n+q}$ and $N^n$ be smooth manifolds of dimensions $n+q$ and $n$ 
respectively. Let $p \in Q^{n+q}$ be a singular point of 
a smooth map $f \co Q^{n+q} \to N^{n}$. The smooth map $f$  has a {\it fold 
singularity of index $\la$} at the singular point $p$ if we can write $f$ in some local coordinates around $p$  
and $f(p)$ in the form 
\[  
f(x_1,\ldots,x_{n+q})=(x_1,\ldots,x_{n-1}, -x_n^2 - \cdots -x_{n+\la-1}^2 + x_{n+\la}^2 + \cdots x_{n+q}^2)
\] 
for some $\la$ $(0 \leq \la \leq q+1)$ (the index $\la$ is well-defined if
we consider that $\la$ and $q+1-\la$ represent the same index). 

A smooth map $f \co Q^{n+q} \to N^{n}$ is called a {\it fold map} if $f$ has only 
fold singularities.

A smooth map $f \co Q^{n+q} \to N^n$ 
  has a {\it definite fold
singularity} at a fold singularity $p \in Q^{n+q}$ if $\la = 0$ or $\la = q+1$,
otherwise $f$ has an {\it indefinite fold singularity of index $\la$}
at the fold singularity $p \in Q^{n+q}$.

Let $S_{\la}(f)$ denote the set of fold singularities of index $\la$ of $f$ in $Q^{n+q}$.
Note that $S_{\la}(f) = S_{q+1-\la}(f)$.
 Let $S_f$ denote the set $\bigcup_{\la} S_{\la}(f)$.

Note that the set $S_f$ is an ${(n-1)}$-dimensional submanifold of the manifold
$Q^{n+q}$.

Note that each connected component of the manifold $S_f$ has its own index $\la$ if
we consider that $\la$ and $q+1-\la$ represent the same index. 

Note that for a fold map $f \co Q^{n+q} \to \R^{n}$ and for an index $\la$ ($0 \leq \la \leq \lfloor (q-1)/2 \rfloor$ or
 $q+1-\lfloor (q-1)/2 \rfloor \leq \la \leq q+1$)
the codimension one immersion $f \mid_{S_{\la}(f)} \co S_{\la}(f) \to \R^n$ 
of the singular set of index $\la$ $S_{\la}(f)$ has a canonical framing 
(i.e., trivialization of the normal bundle) by identifying canonically the set of 
fold singularities of index $\la$ 
($0 \leq \la \leq \lfloor (q-1)/2 \rfloor$ or
 $q+1-\lfloor (q-1)/2 \rfloor \leq \la \leq q+1$)
of the map $f$ with   
the fold germ 
$(x_1,\ldots,x_{n+q}) \mapsto (x_1,\ldots,x_{n-1}, -x_n^2 - \cdots -x_{n+\la-1}^2 + x_{n+\la}^2 + \cdots x_{n+q}^2)$
$(0 \leq \la \leq \lfloor (q-1)/2 \rfloor)$
(if
we consider that $\la$ and $q+1-\la$ represent the same index), see, for example, \cite{Sa1}.

If $f \co Q^{n+q} \to N^n$ is a fold map in general position, then 
the map $f$
restricted to the singular set $S_f$ is a general positional
 codimension one immersion  into the target manifold $N^n$.
 
Since every fold map is in general position after a small perturbation, 
and we study maps under the equivalence relations {\it cobordism} and {\it bordism} 
(see Definitions~\ref{cobdef} and \ref{borddef} respectively),
in this paper we can restrict ourselves to studying fold maps which are 
in general position.
Without mentioning we suppose that a fold map $f$ is in general position.

\subsection{Stein factorization}
 
We use the notion of the Stein factorization of a smooth map 
$f\co Q^q \to N^n$, where $Q^q$ and $N^n$ are smooth manifolds 
of dimensions $q$ and $n$ respectively $(q \geq n)$. 
Two points  $p_1,p_2 \in Q^q$ are {\it equivalent} if  
$p_1$ and $p_2$ lie on the same component of an $f$-fiber.
Let $W_f$ denote the quotient space of $Q^q$ with respect
to this equivalence relation and $q_f \co Q^q \to W_f$ the quotient map.
Then there exists a unique continuous map $\Bar{f} \co W_f \to N^n$ such that
$f = \Bar{f} \circ q_f$. The space $W_f$ or the factorization of the 
map $f$ into the composition of $q_f$ and $\Bar{f}$ is called the {\it Stein
factorization} of the map $f$. We call the map $\Bar{f}$ the  
{\it Stein factorization} of the map $f$ as well. Note that if $f$ is a 
generic smooth map of a closed 
$q$-dimensional manifold into $N^n$ (for example, if $f$ is a fold map in general position),
 then its Stein factorization $W_f$ is a  
compact $n$-dimensional CW complex.

\subsection{Cobordisms of fold maps}\label{kob}

\begin{defn}\label{cobdef} (Cobordism) 
Two fold maps $f_i \co Q_i^{n+q} \to N^n$ $(i=0,1)$  
from closed (oriented) $({n+q})$-dimensional manifolds $Q_i^{n+q}$ $(i=0,1)$ 
into an $n$-dimensional manifold $N^n$ are  
{\it (oriented) cobordant} if 
\begin{enumerate}[a)]
\item
there exists a fold map 
$F \co X^{n+q+1} \to N^n \times [0,1]$ from a compact (oriented) $(n+q+1)$-dimensional 
manifold $X^{n+q+1}$, 
\item
$\del X^{n+q+1} = Q_0^{n+q} \amalg (-)Q_1^{n+q}$ and 
\item
${F \mid}_{Q_0^{n+q} \x [0,\ep)}=f_0 \x
{\mathrm {id}}_{[0,\ep)}$ and ${F \mid}_{Q_1^{n+q} \x (1-\ep,1]}=f_1 \x 
{\mathrm {id}}_{(1-\ep,1]}$, where 
$Q_0^{n+q} \x [0,\ep)$
 and $Q_1^{n+q} \x (1-\ep,1]$ are small collar neighbourhoods of $\del X^{n+q+1}$ with the
identifications $Q_0^{n+q} = Q_0^{n+q} \x \{0\}$ and $Q_1^{n+q} = Q_1^{n+q} \x \{1\}$. 
\end{enumerate}

We call the map $F$ a {\it cobordism} between $f_0$ and $f_1$.
\end{defn} 
This clearly defines an equivalence relation on the set of fold maps 
from closed (oriented) $({n+q})$-dimensional manifolds into an  
$n$-dimensional manifold $N^n$.

Let us denote the (oriented) cobordism classes of fold maps into the Euclidean space $\R^n$ under (oriented) fold cobordisms 
by $\CC ob_{f}^{(O)}(n+q,-q)$. We can define a commutative group operation in the usual way
on this set of (oriented) cobordism classes by the far away disjoint union.

\begin{rem}
Ikegami and Saeki \cite{IS} showed
that the group $\CC ob_f^O(2,-1)$ is isomorphic to $\Z$. 
The author proved that 
$\CC ob_f(2,-1) = \Z \oplus \Z_2$ \cite{Kal}
 and the group $\CC ob_f^O(3,-1)$ is isomorphic to $\Z_2 \oplus \Z_2$ \cite{Kal2}. 
Ikegami computed the 
cobordism groups $\CC ob_{f}^{(O)}(p,1-p)$ of Morse functions of $p$-dimensional manifolds for every $p > 0$ \cite{Ik}.

For $k \geq 0$ there are many results concerning the cobordism groups ${\mathcal Cob}_{\tau}(q, k)$,
where $\tau$ is a set of singularity types, the elements of the group ${\mathcal Cob}_{\tau}(q, k)$ 
are cobordism classes of smooth maps of $q$-dimensional manifolds
into $\R^{q+k}$ with only singularities in $\tau$, and a cobordism between two 
such maps has only singularities in $\tau$. See, for example, \cite{An, EkSzuTer, Ko, RSz, Szucs1, Sz1, Sz2, Szucs2, Sz3, Szucs4}.  
\end{rem}

\subsection{Main results}\label{mainthms}

Now we are ready to state our main theorems.

Recall that the oriented cobordism group of immersions of $(n-1)$-dimensional 
manifolds into $\R^n$ denoted in this paper by $\imm^{\ep^1}(n-1,1)$ is isomorphic to the 
stable homotopy group $\pi_{n-1}^s$ \cite{We}.
In the following theorems we identify the group $\imm^{\ep^1}(n-1,1)$ with the group $\pi_{n-1}^s$.

\begin{thm}\label{ori}
For $n \geq 1$ and $q > 0$ 
the cobordism group $\CC ob_{f}^{(O)}(n+q,-q)$ of 
fold maps of (oriented) $(n+q)$-dimensional manifolds into $\R^n$
contains the direct sum  $\bigoplus_0^{\lfloor (q-1)/2 \rfloor} \pi_{n-1}^s$
 as a direct summand.
 
 This direct sum $\bigoplus_0^{\lfloor (q-1)/2 \rfloor} \pi_{n-1}^s$
 is detected by the homomorphisms $\va_0, \ldots ,\va_{\lfloor (q-1)/2 \rfloor}$,
 where $\va_j \co \CC ob_{f}^{(O)}(n+q,-q) \to \imm^{\ep^1}(n-1,1)$ maps 
 a fold cobordism class $[f]$ to the cobordism class of the framed immersion of the
 singular set of index $j$ of the fold map $f$ $(j = 0, \ldots,\lfloor (q-1)/2 \rfloor)$.
\end{thm}

\begin{thm}\label{unori}
For $n \geq 1$, $q > 0$,
$k \geq 1$ and
 $q = 2k -1$ the cobordism group $\CC ob_{f}(n+q,-q)$ 
 of 
fold maps of unoriented $(n+q)$-dimensional manifolds into $\R^n$
contains the direct sum $\pi_{n}^s(\RP^{\infty}) \oplus \bigoplus_0^{\lfloor (q-1)/2 \rfloor} \pi_{n-1}^s$
as a direct summand.

The direct summand $\pi_{n}^s(\RP^{\infty})$ is detected by the homomorphism
$\psi \co \CC ob_{f}(n+q,-q) \to \imm^{}(n-1,1)$, where $\psi$ maps 
a fold cobordism class $[f]$ to the cobordism class of the immersion of the
 singular set of index $k$ of the fold map $f$. 
\end{thm}

Analogous results about bordism groups of fold maps
can be found in Section~\ref{bord}.

\section{Proof of main theorems}\label{mainthmproof}

\begin{proof}[Proof of Theorem~\ref{ori}]

For $0 \leq j \leq \lfloor (q-1)/2 \rfloor$ let us define homomorphisms 
$\al_j \co \imm^{\ep^1}(n-1,1) \to \CC ob_{f}^{(O)}(n+q,-q)$
and $\va_j \co \CC ob_{f}^{(O)}(n+q,-q) \to \imm^{\ep^1}(n-1,1)$ such that
the compositions $\va_j \circ \al_j$ are the identity maps as follows.



Let $j$ be an integer such that $0 \leq j \leq \lfloor (q-1)/2 \rfloor$. 
Let $h_j \co S^{q+1} \to (-\ep,\ep)$ be a Morse function of the $(q+1)$-dimensional sphere into the
open interval $(-\ep,\ep)$
with four critical points of index 
$0, j, j+1, q+1$, respectively, such that the critical value of index $j$ is zero in the interval $(-\ep,\ep)$.

Let $j$ be an integer such that $0 \leq j \leq \lfloor (q-1)/2 \rfloor$.
\begin{defn}\label{alfadef}
Let $[g \co M^{n-1} \to \R^{n}]$ be an element of the group $\imm^{\ep^1}(n-1,1)$.
We identify the framed fibers of the normal bundle of the immersion $g$ with the interval $(-\ep,\ep)$.
Let $f_j \co M^{n-1} \x S^{q+1} \to \R^{n}$ be a
fold map such that for every $p \in M^{n-1}$ the map $f_j \mid_{\{p\} \x S^{q+1}}$ maps the sphere $\{p\} \x S^{q+1}$ 
as the Morse function $h_j$ onto the framed fiber $(-\ep,\ep)$ over the point $g(p)$ in $\R^n$ of
the normal bundle of the immersion $g$. Let $\al_j([g])$ be the cobordism class of the fold map $f_j$.
In the case of a cobordism between two elements $[g_i]$ $(i=0,1)$ of the group $\imm^{\ep^1}(n-1,1)$ we
can do the same construction in order to obtain a cobordism between the fold cobordism classes $\al_j[g_i]$ $(i=0,1)$.
Therefore this clearly defines a homomorphism  $\al_j \co \imm^{\ep^1}(n-1,1) \to \CC ob_{f}^{(O)}(n+q,-q)$.
\end{defn}


\begin{defn}\label{fidef}
Let $[f \co Q^{n+q} \to \R^{n}]$ be an element of the group $\CC ob_{f}^{(O)}(n+q,-q)$.
Then let $\va_j([f])$ be the cobordism class of the framed immersion $f \mid_{S_j(f)} \co S_j(f) \to \R^n$ 
of the singular set of index $j$
of the fold map $f$. 
This clearly defines a homomorphism  $\va_j \co  \CC ob_{f}^{(O)}(n+q,-q) \to \imm^{\ep^1}(n-1,1)$.
\end{defn}

\begin{lem}
The composition 
\[
\begin{CD}
\imm^{\ep^1}(n-1,1) @> \al_j >> \CC ob_{f}^{(O)}(n+q,-q) @> \va_j >> \imm^{\ep^1}(n-1,1)
\end{CD}
\]
is the identity map.
\end{lem}
\begin{proof}
First let us prove the case $j = 0$.
Let $[g \co M^{n-1} \to \R^{n}]$ be an element of the group $\imm^{\ep^1}(n-1,1)$.
Then the framed immersion $\va_0(\al_0([g]))$ constructed in Definitions~\ref{alfadef} and \ref{fidef} 
has three components which are parallel translations of the immersion $g$, moreover
two of these components have the same framing and the third component has an opposite framing.
The statement of the lemma follows if we note that the sum of
two parallel translations with opposite framing among these three components 
represents the zero cobordism class in
the group  $\imm^{\ep^1}(n-1,1)$. A null-cobordism can be given easily by considering  
the Stein factorization $\bar f_0 \co W_{f_0} \to \R^{n}$ of 
the fold map $f_0 \co M^{n-1} \x S^{q+1} \to \R^{n}$ in Definition~\ref{alfadef} in the case of $j = 0$.
The CW complex $W_{f_0}$ is equal to $M^{n-1} \x W_{h_0}$, where $W_{h_0}$ is the source ``$\Yup$'' of the Stein factorization
$\bar h_0 \co W_{h_0} \to (-\ep,\ep)$
of the Morse function $h_0$, and
the Stein factorization $\bar f_0$ maps the subset $M^{n-1} \x \{t\}$ of $W_{f_0}$ ($t \in W_{h_0}$) 
to $\R^{n}$ as a parallel translation of the immersion $g$.



When $j > 0$, 
clearly by the constructions in Definitions~\ref{alfadef} and \ref{fidef} 
for a given element $[g]$ of the group $\imm^{\ep^1}(n-1,1)$ 
the framed immersion of the singular set of index $j$ of the fold map
$f_j \co M^{n-1} \x S^{q+1} \to \R^n$ equals to $g$. Hence
$\va_j(\al_j([g])) = [g]$ for every element $[g]$ of the group $\imm^{\ep^1}(n-1,1)$.
\end{proof}

By the above lemma in order to complete the proof of Theorem~\ref{ori}
we have to show that the composition 
\[
\begin{CD}
\bigoplus_0^{\lfloor (q-1)/2 \rfloor} \imm^{\ep^1}(n-1,1) @>  \al_0 + \cdots +\al_{\lfloor (q-1)/2 \rfloor}  >> \\ \CC ob_{f}^{(O)}(n+q,-q) @> \bigoplus_{j=0}^{\lfloor (q-1)/2 \rfloor} \va_j >> \bigoplus_0^{\lfloor (q-1)/2 \rfloor} \imm^{\ep^1}(n-1,1)
\end{CD}
\]
denoted by $\ro$ is an isomorphism. 
But the homomorphism $\ro$ is clearly injective and surjective since 
by the definition of the homomorphisms $\al_j$ and $\va_j$ $(0 \leq j \leq \lfloor (q-1)/2 \rfloor)$
the homomorphism
$\ro$ has the form
\[
\ro ( a_0, a_1, \ldots ,a_{\lfloor (q-1)/2 \rfloor}) = ( a_0, a_1+a_0, a_2+a_1, \ldots, a_{\lfloor (q-1)/2 \rfloor}+a_{\lfloor (q-3)/2 \rfloor}).
\]

This completes the proof of Theorem~\ref{ori}.
\end{proof}

\begin{proof}[Proof of Theorem~\ref{unori}]
Let us define two 
homomorphisms $\be \co \imm(n-1,1) \to \CC ob_{f}(n+q,-q)$
and $\psi \co \CC ob_{f}^{}(n+q,-q) \to \imm^{}(n-1,1)$ such that
the composition $\psi \circ \be$ is the identity map as follows.

Let $r \co \RP^{2k} \to (-\ep,\ep)$ be a Morse function of the 
$(2k)$-dimensional projective space into the open interval $(-\ep,\ep)$
such that $r$ has $2k+1$ critical points with indeces $0, \ldots, 2k$, respectively.
We can choose the Morse function $r$ such that there exists a diffeomorphism $\de$ of $\RP^{2k}$ 
with the properties 
$r \circ \de = -r$, $\de \circ \de = {\mathrm {id}}_{\RP^{2k}}$, and 
that the critical value of index $k$ of the function $r$  is $0$ in 
the interval $(-\ep,\ep)$.

\begin{defn}\label{betadef}
Let $[g \co M^{n-1} \to \R^{n}]$ be an element of the group $\imm^{}(n-1,1)$, i.e., 
the normal bundle $\nu(g)$ of the immersion $g$ is a line bundle induced from
the universal line bundle $\ga^1 \to \RP^{\infty}$ which we identify
with $(-\ep,\ep) \x_{\Z_2} S^{\infty} \to \RP^{\infty}$ in the usual way.
Let us induce the normal bundle $\nu(g)$ and 
a manifold $M^{n-1} \bar \x \RP^{2k}$ by considering the commutative diagram 
\[
\begin{CD}
M^{n-1} \bar \x \RP^{2k} @>>> \RP^{2k} \x_{\Z_2} S^{\infty} \\
@VV {\bar r} V  @VV r \x_{\Z_2} {\mathrm {id_{S^{\infty}}}} V \\
\nu(g) @>>> (-\ep,\ep) \x_{\Z_2} S^{\infty} \\
@VVV @VVV \\
M^{n-1} @>>> \RP^{\infty}
\end{CD}
\]
where the $\Z_2$ action on $\RP^{2k}$ %
is the action of the diffeomorphism $\de$ (see also \\ \cite[Example~3.7]{Sa1}).
By the diagram above we obtain a fold map 
$f' = \mu \circ \bar r \co M^{n-1} \bar \x \RP^{2k} \to \R^n$,
where $\mu$ denotes the immersion of the normal bundle $\nu(g)$ into $\R^n$.
Note that for every $p \in M^{n-1}$ the map $f' \mid_{\{p\} \x \RP^{2k}}$ maps the space 
$\{p\} \x \RP^{2k}$
as the Morse function $\pm r$ onto the fiber $(-\ep,\ep)$ over the point $g(p)$ in $\R^n$ of
the normal bundle of the immersion $g$.

Let $\be([g])$ be the cobordism class of the fold map 
$f' \co M^{n-1} \bar \x \RP^{2k} \to \R^{n}$.

In the case of a cobordism between two elements $[g_i]$ $(i=0,1)$ of the group $\imm^{}(n-1,1)$ we
can do the same construction in order to obtain a cobordism between the elements $\be[g_i]$ $(i=0,1)$.
Therefore this clearly defines a homomorphism  $\be \co \imm^{}(n-1,1) \to \CC ob_{f}^{}(n+q,-q)$.
\end{defn}

Note that for any $[g] \in \imm^{}(n-1,1)$ the cobordism classes $\va_j( \be( [g] ))$ 
are zero in the group $\imm^{\ep^1}(n-1,1)$ $(0 \leq j \leq \lfloor (q-1)/2 \rfloor)$
since for the fold map $f'$ in Definition~\ref{betadef}
the immersion 
$f' \mid_{S_j(f')} \co S_j(f') \to \R^n$ $(0 \leq j \leq \lfloor (q-1)/2 \rfloor)$ is a double covering of the immersion 
$f' \mid_{S_k(f')} \co S_k(f') \to \R^n$ and hence it is a boundary of the immersion
of the normal bundle $\nu(f')$ into $\R^n$
(this can be seen easily by considering the Stein factorization of the fold map 
$f' \co M^{n-1} \bar \x \RP^{2k} \to \R^n$).

\begin{defn}\label{psidef}
Let $[f \co Q^{n+q} \to \R^{n}]$ be an element of the group $\CC ob_{f}(n+q,-q)$.
Let $\psi([f])$ be the cobordism class of the immersion $f \mid_{S_k(f)} \co S_k(f) \to \R^n$ 
where $S_k(f)$ denotes the indefinite singular set of index $k$ of the fold map $f$. 
This clearly defines a homomorphism  $\psi \co  \CC ob_{f}^{}(n+q,-q) \to \imm^{}(n-1,1)$.
\end{defn}

\begin{lem}
The composition 
\[
\begin{CD}
\imm^{}(n-1,1) @> \be >> \CC ob_{f}^{}(n+q,-q) @> \psi >> \imm^{}(n-1,1)
\end{CD}
\]
is the identity map.
\end{lem}
\begin{proof}
Clearly by the constructions in Definitions~\ref{betadef} and \ref{psidef} 
for a given element $[g]$ of the group $\imm^{}(n-1,1)$ 
the immersion of the indefinite singular set of index $k$ of the fold map
$f' \co M^{n-1} \bar \x \RP^{2k} \to \R^n$ equals to $g$. Hence
$\psi(\be([g])) = [g]$ for every element of the group $\imm^{}(n-1,1)$.
Details are left to the reader.
\end{proof}

The proof of Theorem~\ref{unori} will be finished if 
we show that 
the composition 
\[
\begin{CD}
\bigoplus_0^{\lfloor (q-1)/2 \rfloor} \imm^{\ep^1}(n-1,1) \oplus \imm^{}(n-1,1) @>  \al_0 + \cdots +\al_{\lfloor (q-1)/2 \rfloor} + \be >> \\ \CC ob_{f}^{}(n+q,-q) @> \bigoplus_{j=0}^{\lfloor (q-1)/2 \rfloor} \va_j \oplus \psi >> \\ \bigoplus_0^{\lfloor (q-1)/2 \rfloor} \imm^{\ep^1}(n-1,1) \oplus \imm^{}(n-1,1)
\end{CD}
\]
denoted by $\tau$ is an isomorphism. 
But the homomorphism $\tau$ is clearly injective and surjective since 
by the definition of the homomorphisms $\al_j$, $\va_j$ $(0 \leq j \leq \lfloor (q-1)/2 \rfloor)$, $\be$ and $\psi$
the homomorphism 
$\tau$ has the form
\[
\tau ( a_0, a_1, \ldots ,a_{\lfloor (q-1)/2 \rfloor}, b) = ( a_0, a_1+a_0, a_2+a_1, \ldots, a_{\lfloor (q-1)/2 \rfloor}+a_{\lfloor (q-3)/2 \rfloor}, b+\iota (a_{\lfloor (q-1)/2 \rfloor})),
\]
where $\iota \co \imm^{\ep^1}(n-1,1) \to \imm^{}(n-1,1)$ denotes the natural forgetting homomorphism.

This completes the proof of Theorem~\ref{unori}.
\end{proof}

\section{Bordisms of fold maps}\label{bord}

There is also a very geometrical equivalence relation
on the set of fold maps $f \co Q^{n+q} \to N^n$ 
from closed (oriented) $({n+q})$-dimensional manifolds $Q^{n+q}$ 
into closed oriented $n$-dimensional manifolds $N^n$.

\begin{defn}\label{borddef} (Bordism)
Two fold maps $f_i \co Q_i^{n+q} \to N_i^n$ $(i=0,1)$  
from closed (oriented) ${(n+q)}$-dimensional manifolds $Q_i^{n+1}$ $(i=0,1)$ 
into closed oriented $n$-dimensional manifolds $N_i^n$ $(i=0,1)$ are  
{\it (oriented) bordant} if 
\begin{enumerate}[a)]
\item
there exists a fold map 
$F \co X^{n+q+1} \to Y^{n+1}$ from a compact (oriented) $(n+q+1)$-dimensional 
manifold $X^{n+q+1}$ to a compact oriented $(n+1)$-dimensional 
manifold $Y^{n+1}$,
\item
$\del X^{n+q+1} = Q_0^{n+q} \amalg (-)Q_1^{n+q}$,
$\del Y^{n+1} = N_0^{n+1} \amalg -N_1^{n+1}$ and
\item
${F \mid}_{Q_0^{n+q} \x [0,\ep)}=f_0 \x
id_{[0,\ep)}$ and ${F \mid}_{Q_1^{n+q} \x (1-\ep,1]}=f_1 \x id_{(1-\ep,1]}$, where $Q_0^{n+q} \x [0,\ep)$
 and $Q_1^{n+q} \x (1-\ep,1]$ are small collar neighbourhoods of $\del X^{n+q+1}$ with the
identifications $Q_0^{n+q} = Q_0^{n+q} \x \{0\}$, $Q_1^{n+q} = Q_1^{n+q} \x \{1\}$. 
\end{enumerate}

We call the map $F$ a {\it bordism} between $f_0$ and $f_1$.
\end{defn}

We can define a commutative group operation on the set of bordism classes 
by $[f_0] + [f_1] = f_0 \amalg f_1 \co  Q_0^{n+q} \amalg Q_1^{n+q} \to N_0^n \amalg N_1^n$
in the usual way.

By changing the target manifold $\R^n$ to an arbitrary oriented $n$-dimensional manifold $N^n$ in the
previous chapters we obtain analogous theorems about the bordism group of fold maps denoted by
$Bor_f^{(O)}(n+q, -q)$.

Let $B\imm^{\csi^k}(n-k,k)$ denote the usual bordism group
of $k$-codimensional immersions into closed oriented $n$-manifolds, whose normal bundles
can be induced from the vector bundle $\csi^k$. Note that this group
$B\imm^{\csi^k}(n-k,k)$ is isomorphic to the $n$th oriented bordism group 
$\Omega_{n}(\Gamma_{\csi^k})$
of the classifying space $\Gamma_{\csi^k}$ for such immersions \cite{Sz1}.

Then we have the following theorems which are analogues of Theorems~\ref{ori} and \ref{unori}.

\begin{thm}
For $n \geq 1$ and $q > 0$ 
the bordism group $Bor_{f}^{(O)}(n+q,-q)$ of 
fold maps of (oriented) $(n+q)$-dimensional manifolds into closed oriented $n$-dimensional 
manifolds
contains the
direct sum  $\bigoplus_0^{\lfloor (q-1)/2 \rfloor} B\imm^{\ep^1}(n-1,1)$
 as a direct summand.
\end{thm}

\begin{thm}
For $n \geq 1$, $q > 0$,
$k \geq 1$ and
 $q = 2k -1$ the 
bordism group $Bor_{f}(n+q,-q)$ of 
fold maps of unoriented $(n+q)$-dimensional manifolds into 
closed oriented $n$-dimensional 
manifolds
contains the
direct sum $B\imm^{\ga^1}(n-1,1) \oplus \bigoplus_0^{\lfloor (q-1)/2 \rfloor} B\imm^{\ep^1}(n-1,1)$
as a direct summand.
\end{thm}


The proofs are analogous to the proofs of Theorems~\ref{ori} and \ref{unori}, respectively.

\end{spacing}
\end{document}